\documentclass[letterpaper,11pt]{article}

\usepackage[T1]{fontenc}
\usepackage{pgfplots}
\usepgfplotslibrary{fillbetween}

\usepackage{geometry}
\usepackage[hidelinks]{hyperref}

\usepackage{amsthm}
\usepackage{amsmath}
\usepackage{amsfonts}
\usepackage{amssymb}
\usepackage{mathdots}

\usepackage{times}
\usepackage{bm}
\usepackage{dutchcal}

\usepackage{tikz-cd}
\usepackage{fancybox}
\usepackage{caption}
\usepackage{subcaption}
\usepackage{graphicx}

\usepackage{tikz}
\usetikzlibrary{arrows}

\usepackage{todonotes}

\DeclareMathOperator{\diag}{diag}

\DeclareMathOperator{\Tr}{Tr}

\DeclareMathOperator{\Ad}{Ad}
\DeclareMathOperator{\ad}{ad}

\newcommand{\n}{^{-1}}

\newcommand\lie{\mathfrak}

\newcommand{\hhh}{\lie{h}}

\newcommand{\kk}{\lie{k}}

\newcommand\bb{\mathbb}

\newcommand\R{\bb{R}} 
\newcommand\CC{\bb{C}}

\usepackage{pstricks}

\newcommand\parallelogram[1][2]{%
    \psset{unit=#1 pt}
    \begin{pspicture}(4,3)
        \pspolygon(0,0)(3,0)(4,3)(1,3)
    \end{pspicture}}
    
\newcommand\mtrap[1][2]{%
    \psset{unit=#1pt}
    \begin{pspicture}(4,3)
        \pspolygon(0,0)(4,0)(3,3)(1,3)
    \end{pspicture}}
    
\newcommand\wtrap[1][2]{%
    \psset{unit=#1pt}
    \begin{pspicture}(4,3)
        \pspolygon(0,3)(4,3)(3,0)(1,0)
    \end{pspicture}}

\usepackage{theoremref}
\newtheorem{theorem}{Theorem}[section]

\newtheorem{proposition}[theorem]{Proposition}
\newtheorem{lemma}[theorem]{Lemma}

\theoremstyle{definition}

\newtheorem{examplex}{Example}
\newenvironment{example}
  {\pushQED{\qed}\examplex}
  {\popQED\endexamplex}

\theoremstyle{remark}
\newtheorem{remark}[theorem]{Remark}

\newcommand\nnfootnote[1]{%
  \begin{NoHyper}
  \renewcommand\thefootnote{}\footnote{#1}%
  \addtocounter{footnote}{-1}%
  \end{NoHyper}
}

\begin{document}
\title{Local normal forms for multiplicity free $U(n)$ actions on coadjoint orbits}
\author{Jeremy Lane}
\newcommand{\Addresses}{{
  \textsc{Department of Mathematics \& Statistics, McMaster University, Hamilton Hall, 1280 Main Street W, Hamilton, ON, L8S 4K1, Canada}\par\nopagebreak
  \textit{E-mail address}: \texttt{lanej5@math.mcmaster.ca}
}}
\date{\today}
\maketitle

\nnfootnote{\emph{Keywords: coadjoint orbits, multiplicity free spaces, Gelfand-Zeitlin, integrable systems.} }

\begin{abstract}
Actions of $U(n)$ on $U(n+1)$ coadjoint orbits via embeddings of $U(n)$ into $U(n+1)$ are an important family of examples of multiplicity free spaces. They are related to Gelfand-Zeitlin completely integrable systems and multiplicity free branching rules in representation theory. This paper computes the Hamiltonian local normal forms of all such actions, at arbitrary points, in arbitrary $U(n+1)$ coadjoint orbits. The results are described using combinatorics of interlacing patterns; gadgets that describe the associated Kirwan polytopes. 
\end{abstract}

\section{Introduction}

A Hamiltonian action of a compact connected Lie group $K$ on compact symplectic manifold $(M,\omega)$ with an equivariant moment map is a \emph{multiplicity free space} if the ring of $K$-invariant functions $C^{\infty}(M)^K$ is a commutative Poisson subalgebra \cite{GS3}. The moment map of a multiplicity free space identifies the orbit space, $M/K$, with a convex polytope called the \emph{Kirwan polytope} after \cite{kirwan}.  Compact multiplicity free spaces are classified by their Kirwan polytope and the principal isotropy subgroup of the action \cite{knop1}. The local classification of multiplicity free spaces (in a neighbourhood of an orbit) is a crucial step in the proof of the  classification theorem for compact multiplicity free spaces. It is equivalent to the classification of smooth affine spherical varieties for $G = K^\CC$. Smooth affine spherical varieties are classified by their weight monoids \cite{losev}.

One particularly concrete family of examples of multiplicity free spaces is provided by the action of a unitary group, $U(n)$, on a coadjoint orbit of the unitary group $U(n+1)$ via an embedding of $U(n)$ into $U(n+1)$ (Section~\ref{section the spaces}).  The Kirwan polytopes of these spaces  can be described as the set of points $(\mu_1, \dots , \mu_n)\in \R^n$ that satisfy the so-called \emph{interlacing inequalities},
\begin{equation}\label{interlacing inequalities}
	\lambda_1 \geq \mu_1 \geq \lambda_{2} \geq \mu_2 \geq  \dots \geq \mu_n \geq \lambda_{n+1},
\end{equation}
where $\lambda_1, \dots ,\lambda_{n+1} \in \R$ are fixed parameters determined by the coadjoint orbit.  The main result of this paper (Theorem \ref{main theorem}) is the computation of the local classifying data of these spaces at arbitrary points in arbitrary $U(n+1)$ orbits.  This result has two interesting features. First, the classifying data are described in terms of combinatorial gadgets called \emph{interlacing patterns} that encode the combinatorics of the Kirwan polytope (see Section~\ref{section interlacing patterns}). An example of an interlacing pattern is illustrated below. It corresponds to certain points in $U(8)$ coadjoint orbits diffeomorphic to $U(8)/U(2) \times U(1)\times U(2) \times U(1)\times U(1)\times U(1)$. 
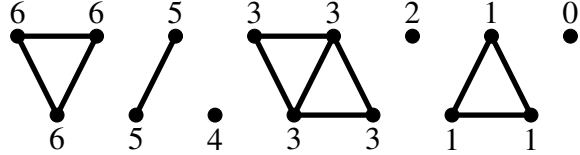
\begin{equation*}
\begin{tikzpicture}[scale =.7,line cap=round,line join=round,>=triangle 45,x=1cm,y=1cm,scale=1.5, every node/.style={scale=1.5}]
\begin{scriptsize}

\foreach \i in {1,...,8}
{
	\draw [fill=black] (\i-1,0) circle (2.5pt);
}
\draw[color=black] (0,0.3) node {6};
\draw[color=black] (1,0.3) node {6};
\draw[color=black] (2,0.3) node {5};
\draw[color=black] (3,0.3) node {3};
\draw[color=black] (4,0.3) node {3};
\draw[color=black] (5,0.3) node {2};
\draw[color=black] (6,0.3) node {1};
\draw[color=black] (7,0.3) node {0};

\foreach \i in {1,...,7}
{
	\draw [fill=black] (\i-.5,-1) circle (2.5pt);
}

\draw[color=black] (.5,-1.3) node {6};
\draw[color=black] (1.5,-1.3) node {5};
\draw[color=black] (2.5,-1.3) node {4};
\draw[color=black] (3.5,-1.3) node {3};
\draw[color=black] (4.5,-1.3) node {3};
\draw[color=black] (5.5,-1.3) node {1};
\draw[color=black] (6.5,-1.3) node {1};

\draw [line width=2pt] (0,0)-- (.5,-1);
\draw [line width=2pt] (1,0)-- (.5,-1);
\draw [line width=2pt] (0,0)-- (1,0);

\draw [line width=2pt] (2,0)-- (1.5,-1);

\draw [line width=2pt] (3,0)-- (3.5,-1);
\draw [line width=2pt] (3,0)-- (4,0);
\draw [line width=2pt] (4,0)-- (3.5,-1);
\draw [line width=2pt] (4,0)-- (4.5,-1);
\draw [line width=2pt] (4.5,-1)-- (3.5,-1);

\draw [line width=2pt] (5.5,-1)-- (6.5,-1);
\draw [line width=2pt] (5.5,-1)-- (6,0);
\draw [line width=2pt] (6.5,-1)-- (6,0);

\end{scriptsize}
\end{tikzpicture}
\end{equation*}
The second interesting feature is the proof (given in Section~\ref{proof of theorem section}). Rather than using the classification of smooth affine spherical varieties, the classifying data are computed directly by elementary means. Following several standard reductions, the main step in this proof is the explicit computation of the isotropy representations (Section~\ref{section computation of isotropy rep}).  It is shown that they are certain products of standard representations and trivial representations of factors of the isotropy subgroup, which has a block diagonal form. The block diagonal factors of the isotropy subgroup that act by standard representations correspond to ``parallelogram shapes'' that appear in the interlacing pattern. For example, the isotropy subgroup corresponding to the interlacing pattern above is $U(1) \times U(1) \times 1 \times U(2) \times U(1)$ and the isotropy representation is $\{0\} \oplus \CC \oplus\{0\} \oplus \CC^2\oplus\{0\}$ (see Example \ref{example interlacing pattern}). The computation of this representation relies on the relationship between the combinatorics of interlacing patterns and divisibility properties of characteristic polynomials of certain Hermitian matrices. 

Motivation for this work is provided by the \emph{Gelfand-Zeitlin}\footnote{Also spelled Gelfand-Cetlin and Gelfand-Tsetlin.} commutative completely integrable systems \cite{GS1}. Although Gelfand-Zeitlin systems have been studied extensively in recent years (see e.g.\ \cite{ALL,BMZ,CKO,L2}), very little is known about their local normal forms as integrable systems near singular fibers (see Example \ref{alamiddine example}). An ongoing program aims to use the results of this paper to prove topological and symplectic local normal forms for Gelfand-Zeitlin systems.   The multiplicity free spaces studied in this paper, as well as the associated Gelfand-Zeitlin systems, have analogues for orthogonal groups and orthogonal coadjoint orbits. The local models of those multiplicity free spaces can be computed in a similar fashion.

The author would like to thank Yael Karshon who some years ago provided him with  notes from a lecture on Gelfand-Zeitlin systems by N.T. Zung that inspired  this paper. The author would also like to thank the Fields Institute and the organizers of the thematic program on Toric Topology and Polyhedral Products for the support of a Fields Postdoctoral Fellowship during writing of this paper.

\section{Hamiltonian group actions and local normal forms}\label{background section}

This section fixes conventions, notation, and recalls the statement of the Marle-Guillemin-Sternberg local normal form. Standard references are \cite{audin,GS4} modulo conventions.

\subsection{Hamiltonian group actions}

Let $K$ be a connected Lie group. Denote its Lie algebra by $\kk$, the dual vector space by $\kk^*$, and the dual pairing by $\langle\cdot ,\cdot\rangle$. Let $\Ad$ and $\Ad^*$ denote the adjoint and coadjoint actions respectively, i.e.~$\langle\Ad_k^*\xi,X\rangle = \langle \xi, \Ad_{k\n} X\rangle$ for $k \in K$, $\xi \in \kk^*$, and $X \in \kk$. Given a left action of $K$ on a manifold $M$, the fundamental vector field of $X \in \kk$ is
\[
	\underline{X}_p = \frac{d}{dt}\bigg\vert_{t=0}\exp(tX)\cdot p,\qquad p\in M.
\]

	Let $(M,\omega)$ a symplectic manifold. A left action of $K$ on $M$ is \emph{Hamiltonian} if there exists an equivariant map $\Phi\colon M \to \kk^*$ such that 
	\[
		\iota_{\underline{X}}\omega = d\langle \Phi,X\rangle.
	\]
	A map $\Phi$ with this property is called a \emph{moment map}. The tuple $(M,\omega,\Phi)$ is a \emph{Hamiltonian $K$-manifold}.
Hamiltonian $K$-manifolds $(M,\omega,\Phi)$ and $(M',\omega',\Phi')$ are \emph{isomorphic} if there exists a $K$-equivariant, symplectic diffeomorphism $\varphi\colon (M,\omega) \to (M',\omega')$ such that $\Phi' \circ \varphi = \Phi$.

\begin{example}[Coadjoint orbits]\label{coadjoint orbits example}
	Let $\mathcal{O} \subset \kk^*$ an orbit of the coadjoint action of $K$.  Given $\xi \in \mathcal{O}$, the tangent space $T_\xi \mathcal{O} \subset \kk^*$ is the set of elements of the form $\ad_X^*\xi$, $X\in \kk$. The \emph{Kostant-Kirillov-Souriau} symplectic form $\omega_{\mathrm{KKS}}$ on $\mathcal{O}$ is defined pointwise by the formula
	\[
		(\omega_{\mathrm{KKS}})_\xi(\ad_X^*\xi,\ad_Y^*\xi) = \langle \xi, [X,Y]\rangle.
	\]
	The inclusion map $\iota \colon \mathcal{O} \to \kk^*$ is a moment map for the coadjoint action of $K$ on $(\mathcal{O},\omega_{\mathrm{KKS}})$.  
\end{example}

\begin{example}[Homomorphisms]\label{homomorphism example}
	Let $(M,\omega,\Phi)$ a Hamiltonian $K$-manifold, $H$ a Lie group, and $\varphi\colon H \to K$ a Lie group homomorphism.  Let $(d\varphi)^*\colon \kk^* \to \hhh^*$ denote the linear map dual to $d\varphi\colon \hhh \to \kk$.  Then the action of $H$ on $M$ defined via the action of $K$ and the homomorphism $\varphi$ is Hamiltonian and $(d\varphi)^* \circ \Phi$ is a moment map.
\end{example}

Let $U(n)$ denote the group of $n\times n$ unitary matrices, with Lie algebra $\mathfrak{u}(n)$, and let $\mathcal{H}_n$ denote the set of $n\times n$ Hermitian matrices, $X = X^\dagger$, where $X \mapsto X^\dagger$ denotes conjugate transpose. Fix the isomorphism 
\begin{equation}\label{identification}
	\mathcal{H}_n \to \mathfrak{u}(n)^*, \quad X \mapsto \left(A \mapsto \frac{1}{\sqrt{-1}}\Tr(XA)\right).
\end{equation}
It is equivariant with respect to the action of $U(n)$ on $\mathcal{H}_n$ by conjugation, $k\cdot X = kXk^\dagger$. 

\begin{example}[Representations]\label{unitary representation example}
	Identify $\CC^n \cong M_{n\times 1}(\CC)$.  The standard symplectic form on $\CC^n$ is 
	\begin{equation}\label{omega std definition}
		\omega_{\mathrm{std}}(\mathbf{x},\mathbf{y}) = \frac{1}{2\sqrt{-1}}(\mathbf{x}^\dagger\mathbf{y} - \mathbf{y}^\dagger\mathbf{x}), \quad \mathbf{x},\mathbf{y} \in M_{n\times 1}(\CC).
	\end{equation}
	The action of $U(n)$ on $\CC^n$ by the standard representation is Hamiltonian with moment map
	\begin{equation}\label{standard representation moment map}
		\Phi(\mathbf{x}) = -\frac{1}{2}\mathbf{x}\mathbf{x}^\dagger.
	\end{equation}
	More generally, suppose that $V$ is a real vector space equipped with a linear symplectic form $\omega_V$. Let $\rho\colon K \to Sp(V,\omega_V)$ be a representation of $K$ on $V$ by symplectic transformations. Then the action of $K$ on $(V, \omega_V)$ defined by $\rho$ is Hamiltonian with moment map $\Phi_V$ defined by the condition
	\begin{equation}\label{unitary action moment map}
		\frac{1}{2}\omega_V(d\rho(X)\mathbf{v},\mathbf{v}) = \langle\Phi_V(\mathbf{v}),X\rangle, \quad \forall \mathbf{v}\in V. \qedhere
	\end{equation}

\end{example}

\begin{example}[Isotropy representations]\label{isotropy representation example}
Let $(M,\omega,\Phi)$ a Hamiltonian $K$-manifold. Given $p\in M$, let $K\cdot p$ denote the orbit of the action of $K$ through $p$ and let $K_p \leq K$ denote the \emph{isotropy subgroup}; the subgroup of elements that fix $p$. Let $K_{\Phi(p)}$ denote the isotropy subgroup of $\Phi(p)$. Then $K_p \leq K_{\Phi(p)}$. The \emph{symplectic slice at} $p\in M$ is the vector space 
\[
	W_p = T_p(K\cdot p)^\omega/(T_p(K\cdot p)\cap T_p(K\cdot p)^\omega)
\]
where $T_p(K\cdot p)^\omega$ denotes the subspace of elements $X \in T_pM$ such that $\omega_p(X,Y) = 0$ for all $Y\in T_p(K\cdot p)$.  The restriction of $\omega_p$ to $T_p(K\cdot p)^\omega$ descends to a symplectic form on $W_p$ denoted $\overline{\omega}_p$.  The linearization of the action of $K_p$, a.k.a.\ the \emph{isotropy representation}, preserves the subspaces $T_p(K\cdot p)^\omega$ and $T_p(K\cdot p)\cap T_p(K\cdot p)^\omega$, so it descends to an action of $K_p$ on $(W_p,\overline{\omega}_p)$ by symplectic transformations. Thus $(W_p,\overline{\omega}_p,\Phi_W)$ is a Hamiltonian $K_p$-manifold, where $\Phi_W$ is defined as in Example \ref{unitary representation example}.
\end{example}

\subsection{Marle-Guillemin-Sternberg local normal forms}

Given a connected Lie group $K$, \emph{Marle-Guillemin-Sternberg data} (MGS data) is a tuple $(\xi,L,W,\omega_W)$ where $\xi \in \kk^*$, $L$ is a Lie subgroup of $K_\xi$, and $(W,\omega_W)$ is a symplectic vector space equipped with a representation of $L$ by symplectic transformations.

Given MGS data $(\xi,L,W,\omega_W)$, \cite{GS4,MARL} construct a Hamiltonian $K$-manifold, denoted $M(\xi,L,W,\omega_W)$, with the following properties. Let $\mathfrak{m} = \kk_\xi/\mathfrak{l}$ and identify $\mathfrak{m}^*$ with a $L$-invariant complement of $\mathfrak{l}^*$ in $\kk_\xi^*$. As a manifold, $M(\xi,L,W,\omega_W)$ is the total space of the vector bundle
\begin{equation}\label{diffeomorphic description of the model space}
	 K \times_{L} (\mathfrak{m}^* \times W) \to K/L
\end{equation}
associated to the principal bundle $L\to K \to K/L$ and the representation $\mathfrak{m}^* \times W$.
The symplectic structure on $M(\xi,L,W,\omega_W)$ is determined by the data $(\xi,L,W,\omega_W)$ (see \cite{GS4,GS5,MARL} for more details).  With respect to this diffeomorphic description of $M(\xi,L,W,\omega_W)$, the Hamiltonian action of $K$ and the corresponding moment map are  
\begin{equation}\label{normal form action and moment map}
	\begin{split}
		k'\cdot[k,\eta,w] & = [k'k,\eta,w],\\ 
		\Phi([k,\eta,w]) & = \Ad_k^*(\eta + \Phi_W(w)+ \xi).
	\end{split}
\end{equation}

Let $(M,\omega,\Phi)$ be a Hamiltonian $K$-manifold. The \emph{Marle-Guillemin-Sternberg data of a point} $p\in M$ is $(\Phi(p),K_p,W_p,\overline{\omega}_p)$, where $K_p$ is the isotropy subgroup of $p$ and $(W_p,\overline{\omega}_p)$ is the symplectic slice at $p$ equipped with the isotropy representation of $K_p$ as described in Example \ref{isotropy representation example}.

\begin{theorem}[Marle-Guillemin-Sternberg local normal forms]\label{mgs theorem}\cite{GS5,MARL}
	Let $(M,\omega,\Phi)$ a Hamiltonian $K$-manifold. For all $p\in M$ there exists $K$-invariant neighbourhoods $U \subset M$ of the orbit $K\cdot p$ and $U' \subset M(\Phi(p),K_p,W_p,\overline{\omega}_p)$ of the orbit $K\cdot [e,0,0]$ and an isomorphism of Hamiltonian $K$-manifolds $\varphi\colon U \to U'$ such that $\varphi(p) = [e,0,0]$.
\end{theorem}

	Hamiltonian $K$-manifolds $(M,\omega,\Phi)$ and $(M',\omega',\Phi')$ 
	are \emph{equivalent} if there exists an automorphism $\psi$ of $K$, a symplectomorphism $F \colon (M,\omega) \to (M',\omega')$, and an $\Ad_K^*$-fixed element $\xi \in \kk^*$ such that:
	\begin{enumerate}
		\item $\psi(k)\cdot F(m) = F(k\cdot m)$, and
		\item $\Phi + \xi = (d\psi)^* \circ \Phi' \circ F$.
	\end{enumerate}

Marle-Guillemin-Sternberg data $(\xi,L,W,\omega_W)$ and $(\xi',L',W',\omega_{W'})$ for $K$ are \emph{equivalent} if the corresponding model spaces  are equivalent as Hamiltonian $K$-manifolds.  
For instance, if $p$ and $p'$ are in the same $K$-orbit, then the MGS data of $p$ and $p'$ are equivalent.

\section{Statement of the main theorem}\label{statement of theorem section}

The following notation will be useful in the remainder of the paper. Given a sequence of real numbers $\underline{\tau} = (\tau_1, \dots ,\tau_n)$, let $[\underline{\tau}]$ denote the set of elements in $\underline{\tau}$. Let $\underline{\tau}_i$ denote the $i$th element of $[\underline{\tau}]$ in decreasing order. Let $m(\underline{\tau})$ denote the size of $[\underline{\tau}]$. Let $n_\tau(\underline{\tau})$ denote the number of times $\tau$ occurs in $\underline{\tau}$. Let $n_i(\underline{\tau})$ denote the number of times $\underline{\tau}_i$ occurs in $\underline{\tau}$.

\subsection{Multiplicity free $U(n)$ actions on $U(n+1)$ coadjoint orbits}\label{section the spaces}

Given a non-increasing sequence of real numbers $\underline{\lambda} = (\lambda_1,\dots, \lambda_{n+1})$, let $\mathcal{O}_\Lambda$ denote the set of matrices in $\mathcal{H}_{n+1}$ with eigenvalues $\lambda_1, \dots , \lambda_{n+1}$. Then $\mathcal{O}_\Lambda$ is the orbit of 
\begin{equation}\label{lambda}
	\Lambda := \left(\begin{array}{ccc}
		\lambda_1 & & \\
		 & \ddots & \\
		 & & \lambda_{n+1}
	\end{array}\right)
\end{equation}
under the action of $U(n+1)$ by conjugation and the map $k \mapsto k\lambda k^\dagger$ descends to a $U(n+1)$-equivariant diffeomorphism
\begin{equation}
	U(n+1) / U(n_1(\underline{\lambda})) \times \dots \times U(n_{m(\underline{\lambda})}(\underline{\lambda})) \to \mathcal{O}_\Lambda.
\end{equation}
The map \eqref{identification} defines a $U(n)$-equivariant diffeomorphism of $\mathcal{O}_\Lambda$ with a coadjoint orbit of $U(n+1)$.  Let $\omega_\Lambda$ denote the symplectic form on $\mathcal{O}_\Lambda$ defined by this identification and the Kostant-Kirillov-Souriau symplectic form defined in Example \ref{coadjoint orbits example}.  For all $p \in \mathcal{O}_\Lambda$,
\begin{equation}\label{symplectic form equation}
	(\omega_\Lambda)_p([X,p],[Y,p])  = \frac{1}{\sqrt{-1}}\Tr\left(p[X,Y]\right) \quad \forall  X,Y \in \mathfrak{u}(n+1).
\end{equation}
With respect to \eqref{identification}, $(\mathcal{O}_\Lambda,\omega_\Lambda,\iota\colon \mathcal{O}_\Lambda \to \mathcal{H}_{n+1})$ is a Hamiltonian $U(n+1)$-manifold, where $\iota$ denotes inclusion. Let $K = U(n)$ and let $\varphi\colon K \to U(n+1)$ be an embedding of $K$ as a Lie subgroup of $U(n+1)$. With respect to the identification \eqref{identification}, $(d\varphi)^*$ is a linear projection $\mathcal{H}_{n+1} \to  \mathcal{H}_n$. By Example \ref{homomorphism example}, $(\mathcal{O}_\Lambda,\omega_\Lambda,\Phi)$ is a Hamiltonian $K$-manifold  with moment map 
\begin{equation}\label{un action moment map}
	\Phi = (d\varphi)^*\circ \iota\colon \mathcal{O}_\Lambda \to \mathcal{H}_{n}.
\end{equation}
	It is well-known that $(\mathcal{O}_\Lambda,\omega_\Lambda,\Phi)$ are multiplicity free spaces for all possible choices of $\underline{\lambda}$ and $\varphi$ (this follows from Lemma~\ref{risquare} below).
\subsection{Interlacing patterns}\label{section interlacing patterns}

Let $\underline{\lambda} = (\lambda_1, \dots,\lambda_{n+1})$ and $\underline{\mu} = (\mu_1, \dots, \mu_n)$ be non-increasing sequences of numbers that satisfy the interlacing inequalities \eqref{interlacing inequalities}. The inequalities \eqref{interlacing inequalities} are represented by attaching labels to a fixed set of $2n+1$ vertices arranged on a triangular grid as illustrated by the following example. 
\begin{equation}
\begin{tikzpicture}[scale =.7,line cap=round,line join=round,>=triangle 45,x=1cm,y=1cm,scale=1.5, every node/.style={scale=1.5}]

\begin{scriptsize}
\foreach \i in {1,...,8}
{
	\draw [fill=black] (\i-1,0) circle (2.5pt);
	\draw[color=black] (\i-1,0.3) node {$\lambda_{\i}$};
}
\foreach \i in {1,...,7}
{
	\draw [fill=black] (\i-.5,-1) circle (2.5pt);
	\draw[color=black] (\i-.5,-1.3) node {$\mu_{\i}$};
}

\end{scriptsize}
\end{tikzpicture}
\end{equation}
If a vertex labelled $x$ appears to the left of a vertex labelled $y$, then $x \geq y$. The labels on the top row correspond to $\underline{\lambda}$ and the labels on the bottom row correspond to $\underline{\mu}$.

The (labelled) \emph{interlacing pattern} of a pair of sequences $(\underline{\lambda},\underline{\mu})$ that satisfy \eqref{interlacing inequalities} is the labelled undirected plane graph obtained by adding straight edges to the diagram above according to the following rule:  two vertices are connected by an edge iff they are nearest neighbours and their labels are equal. For example, the following is the interlacing pattern of $(\underline{\lambda},\underline{\mu})$ where $\underline{\lambda} = (6,6,5,3,3,2,1,0)$ and $\underline{\mu} = (6,5,4,3,3,1,1)$.
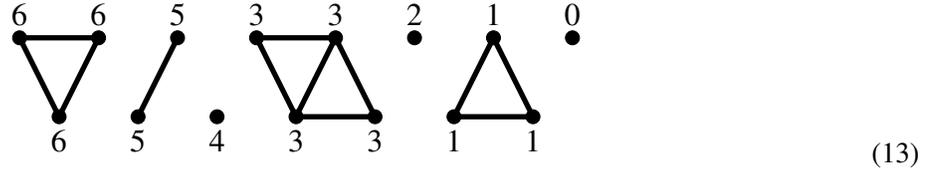
\begin{equation}\label{fig2}
\begin{tikzpicture}[scale =.7,line cap=round,line join=round,>=triangle 45,x=1cm,y=1cm,scale=1.5, every node/.style={scale=1.5}]
\begin{scriptsize}

\foreach \i in {1,...,8}
{
	\draw [fill=black] (\i-1,0) circle (2.5pt);
}
\draw[color=black] (0,0.3) node {6};
\draw[color=black] (1,0.3) node {6};
\draw[color=black] (2,0.3) node {5};
\draw[color=black] (3,0.3) node {3};
\draw[color=black] (4,0.3) node {3};
\draw[color=black] (5,0.3) node {2};
\draw[color=black] (6,0.3) node {1};
\draw[color=black] (7,0.3) node {0};

\foreach \i in {1,...,7}
{
	\draw [fill=black] (\i-.5,-1) circle (2.5pt);
}

\draw[color=black] (.5,-1.3) node {6};
\draw[color=black] (1.5,-1.3) node {5};
\draw[color=black] (2.5,-1.3) node {4};
\draw[color=black] (3.5,-1.3) node {3};
\draw[color=black] (4.5,-1.3) node {3};
\draw[color=black] (5.5,-1.3) node {1};
\draw[color=black] (6.5,-1.3) node {1};

\draw [line width=2pt] (0,0)-- (.5,-1);
\draw [line width=2pt] (1,0)-- (.5,-1);
\draw [line width=2pt] (0,0)-- (1,0);

\draw [line width=2pt] (2,0)-- (1.5,-1);

\draw [line width=2pt] (3,0)-- (3.5,-1);
\draw [line width=2pt] (3,0)-- (4,0);
\draw [line width=2pt] (4,0)-- (3.5,-1);
\draw [line width=2pt] (4,0)-- (4.5,-1);
\draw [line width=2pt] (4.5,-1)-- (3.5,-1);

\draw [line width=2pt] (5.5,-1)-- (6.5,-1);
\draw [line width=2pt] (5.5,-1)-- (6,0);
\draw [line width=2pt] (6.5,-1)-- (6,0);

\end{scriptsize}
\end{tikzpicture}
\end{equation}
Three types of connected components can occur in interlacing patterns: \wtrap\emph{-shapes}, \mtrap\emph{-shapes}, and \parallelogram\emph{-shapes}. In the example \eqref{fig2}: the components labelled 6, 2, and 0 are \wtrap-shapes, the components labelled 4 and 1 are \mtrap-shapes, and the components labelled 5 and 3 are \parallelogram-shapes. By convention, an isolated vertex on the top row is a \wtrap-shape and an isolated vertex on the bottom row is a \mtrap-shape.

If $\underline{\lambda} = (\lambda_1, \dots,\lambda_{n+1})$ is fixed, then the set of pairs $(\underline{\lambda},\underline{\mu})$ that satisfy \eqref{interlacing inequalities} (equivalently, the set of labelled interlacing patterns whose labels on the top row are given by $\underline{\lambda}$) is in bijection with elements of the polytope
\[
	\Delta_{\underline{\lambda}} := \{ \underline{\mu} = (\mu_1, \dots, \mu_n) \in \R^n \mid (\underline{\lambda},\underline{\mu}) \text{ satisfies \eqref{interlacing inequalities}}\}.
\]

Given $(\mathcal{O}_\Lambda,\omega_\Lambda,\Phi)$ as in the previous section, a point $p \in \mathcal{O}_\Lambda$ determines a pair $(\underline{\lambda},\underline{\mu})$ that satisfies \eqref{interlacing inequalities}, where  $\underline{\mu} = (\mu_1,\ldots , \mu_n)$ denotes the eigenvalues of $\Phi(p)$ arranged in non-increasing order. Thus, every $p \in \mathcal{O}_\Lambda$ has an associated labelled interlacing pattern. As observed in \cite{GS1}, the polytope $\Delta_{\underline{\lambda}}$ defined above is the Kirwan polytope of $(\mathcal{O}_\Lambda,\omega_\Lambda,\Phi)$, i.e.
\[
	\Delta_{\underline{\lambda}} = \{ (\mu_1, \dots, \mu_n) \in \R^n \mid \mu_1 \geq \dots \geq \mu_n, \, \exists p \in \mathcal{O}_\Lambda \text{ with eigenvalues }	\mu_1, \dots, \mu_n\}.
\]

The notation $\substack{ \lambda\in [\underline{\lambda}]\\ \text{\wtrap-shape}}$
denotes the set of all $\lambda\in [\underline{\lambda}]$ such that the connected component of the interlacing pattern of $(\underline{\lambda},\underline{\mu})$ labelled by $\lambda$ is a \wtrap-shape. Similar notation is used for other sets. For example, any pair $(\underline{\lambda},\underline{\mu})$ satisfying \eqref{interlacing inequalities} satisfies the identity
\begin{equation}
	\sum_{i=1}^{n+1} \lambda_i - \sum_{i=1}^n\mu_i = \sum_{\substack{ \lambda\in [\underline{\lambda}]\\ \text{\wtrap-shape}}} \lambda - \sum_{\substack{ \mu\in [\underline{\mu}]\\ \text{\mtrap-shape}}} \mu.
\end{equation}

\begin{remark}
An \emph{unlabelled interlacing pattern} is an undirected plane graph that can be obtained from a labelled interlacing pattern by erasing the labels. In other words, the edges in an unlabelled interlacing pattern must correspond to a configuration of equalities and strict inqualities that is allowed by \eqref{interlacing inequalities}. For instance, the following is an unlabelled interlacing pattern. 
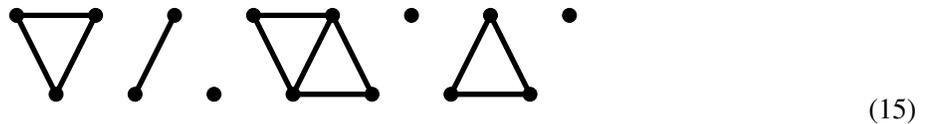
\begin{equation}\label{fig3}
\begin{tikzpicture}[scale =.7,line cap=round,line join=round,>=triangle 45,x=1cm,y=1cm,scale=1.5, every node/.style={scale=1.5}]
\begin{scriptsize}

\foreach \i in {1,...,8}
{
	\draw [fill=black] (\i-1,0) circle (2.5pt);
}
\draw[color=black] (0,0.3) ;
\draw[color=black] (1,0.3) ;
\draw[color=black] (2,0.3) ;
\draw[color=black] (3,0.3) ;
\draw[color=black] (4,0.3) ;
\draw[color=black] (5,0.3) ;
\draw[color=black] (6,0.3) ;
\draw[color=black] (7,0.3) ;

\foreach \i in {1,...,7}
{
	\draw [fill=black] (\i-.5,-1) circle (2.5pt);
}

\draw[color=black] (.5,-1.3) ;
\draw[color=black] (1.5,-1.3) ;
\draw[color=black] (2.5,-1.3) ;
\draw[color=black] (3.5,-1.3) ;
\draw[color=black] (4.5,-1.3) ;
\draw[color=black] (5.5,-1.3) ;
\draw[color=black] (6.5,-1.3) ;

\draw [line width=2pt] (0,0)-- (.5,-1);
\draw [line width=2pt] (1,0)-- (.5,-1);
\draw [line width=2pt] (0,0)-- (1,0);

\draw [line width=2pt] (2,0)-- (1.5,-1);

\draw [line width=2pt] (3,0)-- (3.5,-1);
\draw [line width=2pt] (3,0)-- (4,0);
\draw [line width=2pt] (4,0)-- (3.5,-1);
\draw [line width=2pt] (4,0)-- (4.5,-1);
\draw [line width=2pt] (4.5,-1)-- (3.5,-1);

\draw [line width=2pt] (5.5,-1)-- (6.5,-1);
\draw [line width=2pt] (5.5,-1)-- (6,0);
\draw [line width=2pt] (6.5,-1)-- (6,0);

\end{scriptsize}
\end{tikzpicture}
\end{equation}
On the other hand, the following is not an unlabelled interlacing pattern.
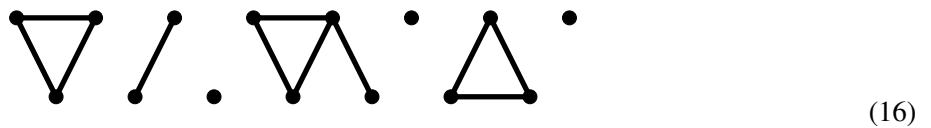
\begin{equation}\label{fig4}
\begin{tikzpicture}[scale =.7,line cap=round,line join=round,>=triangle 45,x=1cm,y=1cm,scale=1.5, every node/.style={scale=1.5}]
\begin{scriptsize}

\foreach \i in {1,...,8}
{
	\draw [fill=black] (\i-1,0) circle (2.5pt);
}
\draw[color=black] (0,0.3) ;
\draw[color=black] (1,0.3) ;
\draw[color=black] (2,0.3) ;
\draw[color=black] (3,0.3) ;
\draw[color=black] (4,0.3) ;
\draw[color=black] (5,0.3) ;
\draw[color=black] (6,0.3) ;
\draw[color=black] (7,0.3) ;

\foreach \i in {1,...,7}
{
	\draw [fill=black] (\i-.5,-1) circle (2.5pt);
}

\draw[color=black] (.5,-1.3) ;
\draw[color=black] (1.5,-1.3) ;
\draw[color=black] (2.5,-1.3) ;
\draw[color=black] (3.5,-1.3) ;
\draw[color=black] (4.5,-1.3) ;
\draw[color=black] (5.5,-1.3) ;
\draw[color=black] (6.5,-1.3) ;

\draw [line width=2pt] (0,0)-- (.5,-1);
\draw [line width=2pt] (1,0)-- (.5,-1);
\draw [line width=2pt] (0,0)-- (1,0);

\draw [line width=2pt] (2,0)-- (1.5,-1);

\draw [line width=2pt] (3,0)-- (3.5,-1);
\draw [line width=2pt] (3,0)-- (4,0);
\draw [line width=2pt] (4,0)-- (3.5,-1);
\draw [line width=2pt] (4,0)-- (4.5,-1);

\draw [line width=2pt] (5.5,-1)-- (6.5,-1);
\draw [line width=2pt] (5.5,-1)-- (6,0);
\draw [line width=2pt] (6.5,-1)-- (6,0);

\end{scriptsize}
\end{tikzpicture}
\end{equation}
If $\underline{\mu}$ and $\underline{\mu}'$ are contained in the relative interior of the same face of $\Delta_{\underline{\lambda}}$, then the unlabelled interlacing patterns of $(\underline{\lambda},\underline{\mu})$ and $(\underline{\lambda},\underline{\mu}')$ are the same. Thus the set of unlabelled interlacing patterns obtained by erasing labels from labelled interlacing patterns of pairs $(\underline{\lambda},\underline{\mu})$, $\underline{\lambda}$ fixed, is in natural bijection with the set of faces of $\Delta_{\underline{\lambda}}$. The partial order on faces of $\Delta_{\underline{\lambda}}$ corresponds to an obvious partial order on the set of all such unlabelled interlacing patterns. Thus, they encode  $\Delta_{\underline{\lambda}}$ as an abstract polytope. It is also straightforward to read the local moment cone of a point $\underline{\mu} \in \Delta_{\underline{\lambda}}$ from the unlabelled interlacing pattern of $(\underline{\lambda},\underline{\mu})$. The intersection of this local moment cone with the standard lattice in $\R^n$ is the weight monoid of the corresponding smooth affine spherical variety that appears in the classification of \cite{knop1}.
\end{remark}

\begin{remark}
	The interlacing patterns described here occur as rows in larger diagrams, also called interlacing patterns, that describe points and faces of Gelfand-Zeitlin polytopes as well as fibers of Gelfand-Zeitlin systems (see e.g.~\cite{ACK,CKO,P,BMZ}). Some authors use an equivalent combinatorial gadget called \emph{ladder diagrams} and introduce terminology such as W-blocks, M-blocks, and N-blocks that is equivalent to the notions of \wtrap\emph{-shapes}, \mtrap\emph{-shapes}, and \parallelogram\emph{-shapes} used here. 
\end{remark}

\subsection{Statement of the main theorem}

Let $K=U(n)$ and let $(\underline{\lambda},\underline{\mu})$ be a pair of non-increasing sequences $\underline{\lambda} = (\lambda_1, \dots ,\lambda_{n+1})$ and $\underline{\mu}= (\mu_1,\dots, \mu_n)$ that satisfy the interlacing inequalities \eqref{interlacing inequalities}. Let $\mathrm{M}:= \text{diag}(\mu_1, \dots ,\mu_n)$.  The stabilizer subgroup $K_\mathrm{M}$ for the conjugation action of $K$ is a block diagonal subgroup isomorphic to $U(n_1(\underline{\mu}))\times \dots \times U(n_{m(\underline{\mu})}(\underline{\mu}))$.
Define 
\begin{equation}\label{topological data defn}
	W_{(\underline{\lambda},\underline{\mu})}  := \bigoplus_{\substack{\mu \in [\underline{\mu}]\\ \text{\parallelogram-shape}}}\CC^{n_\mu(\underline{\mu})},
\end{equation}
and the block-diagonal subgroup 
\begin{equation}
	L_{(\underline{\lambda},\underline{\mu})}   :=  L_1 \times \dots \times L_{m(\underline{\mu})} \leq U(n_1(\underline{\mu}))\times \dots \times U(n_{m(\underline{\mu})}(\underline{\mu})) =  K_\mathrm{M}
\end{equation}
where
\begin{equation}
	L_i = 
		\left\{\left(\begin{array}{c|c} 1 & 0 \\ \hline 0 & k\end{array}\right)\mid k\in  U(n_i(\underline{\mu})-1)\right\}\leq U(n_i(\underline{\mu}))
\end{equation}
if the component of the interlacing pattern of $(\underline{\lambda},\underline{\mu})$ labelled $\underline{\mu}_i$ is a \mtrap-shape, and $L_i  = U(n_i(\underline{\mu}))$ otherwise.  Equip $W_{(\underline{\lambda},\underline{\mu})}$ with the representation of $L_{(\underline{\lambda},\underline{\mu})}$ where the factor $L_i $ acts by the standard representation on the corresponding factor $\CC^{n_i(\underline{\mu})}$ if the component of the interlacing pattern of $(\underline{\lambda},\underline{\mu})$ labelled $\underline{\mu}_i$ is a \parallelogram-shape, and it acts trivially otherwise.  

\begin{example}\label{example interlacing pattern} Consider the interlacing pattern in Figure \ref{fig2}. Then $\mathrm{M} = \diag(6,5,4,3,3,1,1)$, 
\begin{equation}
	\begin{split}
		L_{(\underline{\lambda},\underline{\mu})} & = \left\{ \left(\begin{array}{c|c|c|c|cc}
			k_6 & & &&& \\
			\hline
			 & k_5 &&&& \\
			 \hline
			 & & 1 &&&  \\
			 \hline
			 & & & k_3 && \\
			 \hline
			 & & & & 1 & \\
			 & & & & & k_1
		\end{array}\right) \mid k_6,k_5,k_1 \in U(1), \, k_3 \in U(2)\right\} \\
		W_{(\underline{\lambda},\underline{\mu})} & = \{0\} \oplus \CC \oplus \{0\} \oplus \CC^2 \oplus \{0\}.\\
	\end{split}
\end{equation}
The representation of $L_{(\underline{\lambda},\underline{\mu})}$ on $W_{(\underline{\lambda},\underline{\mu})}$ is $(k_6,k_5,k_3,k_1)\cdot (\mathbf{z}_5,\mathbf{z}_3) = (k_5\mathbf{z}_5,k_3\mathbf{z}_3)$.
\end{example}

For $\mu \in \underline{\mu}$, define $r_\mu\geq0$ such that
\begin{equation}
	r_\mu^2 = -\left(\displaystyle\prod_{\substack{ \lambda\in [\underline{\lambda}]\\ \text{\wtrap-shape}}}\left(\mu - \lambda\right)\right)\left(\displaystyle\prod_{\substack{\tau\in [\underline{\mu}]\\ \text{\mtrap-shape}\\\tau \neq \mu}}\frac{1}{\left( \mu - \tau\right)}\right)
\end{equation}
if the connected component of the interlacing pattern of $(\underline{\lambda},\underline{\mu})$ labelled $\mu$ is a \mtrap-shape, and $r_\mu = 0$ otherwise. If the connected component of the interlacing pattern of $(\underline{\lambda},\underline{\mu})$ labelled $\mu$ is a \mtrap-shape, then $r_\mu^2>0$.

Provided that the component of the interlacing pattern of $(\underline{\lambda},\underline{\mu})$ labelled $\mu = \underline{\mu}_i$ is not a \mtrap-shape, define
\begin{equation}\label{Ci definition}
	C_i := C_\mu := \sum_{i=1}^{n+1}\lambda_i - \sum_{i=1}^n\mu_i  -\mu +\sum_{\substack{\tau \in [\underline{\mu}]\\ \text{\mtrap-shape}}}\frac{r_\tau^2}{\mu - \tau}.
\end{equation}
Finally, define a linear symplectic form on $W_{(\underline{\lambda},\underline{\mu})}$ by the formula
\begin{equation}\label{omegalmu definition}
	\omega_{(\underline{\lambda},\underline{\mu})}(\mathbf{u},\mathbf{w}) := \frac{1}{\sqrt{-1}}\sum_{\substack{\mu \in [\underline{\mu}]\\\text{\parallelogram-shape}}}\frac{-\mathbf{u}_\mu^\dagger \mathbf{w}_\mu + \mathbf{w}_\mu^\dagger\mathbf{u}_\mu}{C_\mu},
\end{equation}
for all $\mathbf{u}, \mathbf{w} \in W_{(\underline{\lambda},\underline{\mu})}$, where $\mathbf{u}_\mu$ denotes the projection of $\mathbf{u}$ to the factor $\CC^{n_\mu(\underline{\mu})}$.

\begin{theorem}\label{main theorem}
	Let $K = U(n)$ and let $(\mathcal{O}_\Lambda,\omega_\Lambda,\Phi)$ be the Hamiltonian $K$-manifold associated to a non-increasing sequence $\underline{\lambda} = (\lambda_1, \dots,\lambda_{n+1})$ and an embedding $\varphi\colon K \to U(n+1)$ as in Section~\ref{section the spaces}. Then, the Marle-Guillemin-Sernberg local normal form data of $p \in \mathcal{O}_\Lambda$ is equivalent to 
	\begin{equation}
		(\mathrm{M},L_{(\underline{\lambda},\underline{\mu})},W_{(\underline{\lambda},\underline{\mu})},\omega_{(\underline{\lambda},\underline{\mu})})
	\end{equation}
	where $(\underline{\lambda},\underline{\mu})$ is determined by $p$ as in Section~\ref{section interlacing patterns} and $\mathrm{M},L_{(\underline{\lambda},\underline{\mu})},W_{(\underline{\lambda},\underline{\mu})},$ and $\omega_{(\underline{\lambda},\underline{\mu})}$ are as defined above.
\end{theorem}

The proof of Theorem \ref{main theorem}, given in Section 4, describes an explicit linear isomorphism between the isotropy representation at $p$ and the symplectic representation $(W_{(\underline{\lambda},\underline{\mu})},\omega_{(\underline{\lambda},\underline{\mu})})$.

\begin{remark}
It is straightforward to check that as $L_{(\underline{\lambda},\underline{\mu})}$-representations,
\begin{equation}\label{subspace m}
	\mathfrak{m}^* \cong \bigoplus_{\substack{\mu \in [\underline{\mu}] \\ \text{\mtrap-shape}}} (\R \times \CC^{n_\mu(\underline{\mu})-1})
\end{equation}
where if the component of the interlacing pattern labelled $\underline{\mu}_i$ is a \mtrap-shape, then the factor $L_i \cong U(n_i(\underline{\mu})-1)$ acts on the corresponding factor $\R \times \CC^{n_i(\underline{\mu})-1}$ as the product of the trivial representation and the standard representation. Otherwise the factor $L_i$ acts trivially. The moment map of the local normal form $M(\mathrm{M},L_{(\underline{\lambda},\underline{\mu})},W_{(\underline{\lambda},\underline{\mu})},\omega_{(\underline{\lambda},\underline{\mu})})$ is easily computed by combining Example \ref{unitary representation example} and \eqref{normal form action and moment map}.
\end{remark}

\begin{example}\label{alamiddine example} 
Let $\lambda_1>\lambda_2>\lambda_3$ and let $p \in \mathcal{O}_\Lambda$ such that the eigenvalues of $\Phi(p)$ are $\mu_1 = \mu_2 = \lambda_2$. The interlacing pattern of $p$ is 
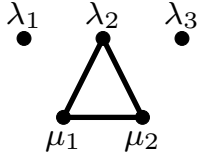
\begin{center}
\begin{tikzpicture}[scale =.7,line cap=round,line join=round,>=triangle 45,x=1cm,y=1cm,scale=1.5, every node/.style={scale=1.5}]
\begin{scriptsize}
\foreach \i in {1,...,3}
{
	\draw [fill=black] (\i-1,0) circle (2.5pt);
}
\draw[color=black] (0,0.3) node {$\lambda_1$};
\draw[color=black] (1,0.3) node {$\lambda_2$};
\draw[color=black] (2,0.3) node {$\lambda_3$};

\foreach \i in {1,...,2}
{
	\draw [fill=black] (\i-.5,-1) circle (2.5pt);
}
\draw[color=black] (.5,-1.3) node {$\mu_1$};
\draw[color=black] (1.5,-1.3) node {$\mu_2$};

\draw [line width=2pt] (1,0)-- (.5,-1);
\draw [line width=2pt] (1,0)-- (1.5,-1);
\draw [line width=2pt] (.5,-1)-- (1.5,-1);

\end{scriptsize}
\end{tikzpicture}
\end{center}
It follows from Theorem \ref{main theorem} that the orbit through $p$ is a Lagrangian $U(2)/U(1)\cong S^3$ and a neighbourhood of this orbit is isomorphic to a neighbourhood of the zero section in $T^*S^3$, equipped with the Hamiltonian action of $U(2)$ by cotangent lift of the action of $U(2)$ on $S^3$.
This particular example was derived by \cite{AL} who used it to show that the Gelfand-Zeitlin systems on regular $U(3)$ coadjoint orbits are isomorphic, in a neighbourhood of this Lagrangian $S^3$ fiber, to an integrable system for the normalized geodesic flow on $T^*S^3$ for the round metric on $S^3$. 
\end{example}

\section{Proof of Theorem \ref{main theorem}}\label{proof of theorem section}

Let $K = U(n)$ and fix an arbitrary non-increasing sequence  $\underline{\lambda} = (\lambda_1, \dots,\lambda_{n+1})$. Several standard reductions are in order. 

First, any two embeddings $K\to U(n+1)$ endow $\mathcal{O}_\Lambda$ with  equivalent Hamiltonian $K$-manifold structures: the restricted coadjoint actions differ by the coadjoint action of an element  $g \in U(n+1)$. Thus, it is sufficient to compute the MGS data with respect to the embedding
\begin{equation}\label{preferred embedding}
	\varphi\colon K \to U(n+1), \quad k \mapsto \left(\begin{array}{c|c}
			1 &  0 \\
			\hline
			0 & k
		\end{array}\right).
\end{equation}
With respect to \eqref{identification}, 
\begin{equation}
	(d\varphi)^*\colon \mathcal{H}_{n+1} \to \mathcal{H}_{n}, \quad (d\varphi)^*(X) = X^{(n)},
\end{equation}
where $X^{(n)}$ denotes the bottom right principal $n\times n$ submatrix of $X$. Thus $\Phi(X) = X^{(n)}$.

Second, it is sufficient to compute the MGS data for points of the form 
\begin{equation}\label{m1}
	p = \left(\begin{array}{c|c}
			c &  \mathbf{z}^\dagger \\
			\hline
			\mathbf{z} & \mathrm{M}
		\end{array}\right) = \left( 
	\begin{array}{c|ccccc}
		c                 & \overline{z}_1 & \overline{z}_2 & \cdots  & \overline{z}_{n-1} & \overline{z}_{n} \\
		\hline
		z_1    &\mu_1 & & & &  \\
		z_2    & &\mu_2 & & & \\
		\vdots            & & & \ddots & &  \\
		z_{n-1}& & & & \mu_{n-1} &  \\
		z_{n}  & & & & & \mu_n \\
	\end{array}
	\right), \quad z_i \in \CC\text{ and } c = \sum_{i=1}^{n+1} \lambda_i - \sum_{i=1}^{n} \mu_i,
\end{equation}
 where $\mu_1\geq \dots \geq \mu_n$. Indeed, every point in $\mathcal{O}_\Lambda$ can be brought to this form by the action $U(n)$, so its MGS data is equivalent to the MGS data of a point of this form.  Note that  $p\in \Phi\n(\mathrm{M})$ if and only if $p$ is of the form \eqref{m1}.
 
Before giving the final reduction, recall from  \cite{GS1} that the condition  $p\in \mathcal{O}_\Lambda$, for $p$ of the form \eqref{m1}, is equivalent to the following equality of characteristic polynomials,
\begin{equation}\label{polynomial}
	\prod_{i=1}^{n+1}\left(x - \lambda_i\right) = (x-c)\prod_{i=1}^n\left( x - \mu_i\right) - \sum_{i=1}^{n}\vert z_i \vert^2\prod_{\substack{j=1\\i\neq j}}^n\left( x - \mu_j\right).
\end{equation}
Re-write $p$ in block form 
\begin{equation}\label{m2}
	p = \left( 
	\begin{array}{c|c|c|c|c}
		c                 & \mathbf{z}_1^\dagger & \mathbf{z}_2^\dagger & \cdots  & \mathbf{z}_{m}^\dagger \\
		\hline
		\mathbf{z}_1    &\underline{\mu}_1I_{n_1(\underline{\mu})} & & &  \\
		\hline
		\mathbf{z}_2    & &\underline{\mu}_2I_{n_2(\underline{\mu})} & &  \\
		\hline
		\vdots            & & & \ddots &   \\
		\hline
		\mathbf{z}_{m}  & & & &  \underline{\mu}_{m}I_{n_{m}(\underline{\mu})} \\
	\end{array}
	\right), \quad \mathbf{z}_i \in M_{n_i(\underline{\mu})\times 1}(\CC).
\end{equation}
where $m=m(\underline{\mu})$. If $\mu = \underline{\mu}_i$, let $\mathbf{z}_\mu = \mathbf{z}_i$ denote the corresponding block. Then \eqref{polynomial} becomes
\begin{equation}\label{polynomial2}
	\prod_{\lambda\in [\underline{\lambda}]}\left(x - \lambda\right)^{n_\lambda(\underline{\lambda})} = (x-c)\prod_{\mu \in [\underline{\mu}]}\left( x - \mu\right)^{n_\mu(\underline{\mu})} - \sum_{\mu \in [\underline{\mu}]}\vert\vert \mathbf{z}_\mu \vert\vert^2(x-\mu)^{n_\mu(\underline{\mu})-1}\prod_{\substack{\tau \in [\underline{\mu}] \\\tau \neq \mu}}\left( x - \tau\right)^{n_\tau(\underline{\tau})}.
\end{equation}

The following lemma is well-known. It's proof is left as an exercise using the fact that $p \in \mathcal{O}_\Lambda$ iff $p$ satisfies \eqref{polynomial2}.
 
\begin{lemma}\label{risquare} \label{un action is multiplicity free} Let $p$ be of the form \eqref{m2}. Then $p\in \mathcal{O}_\Lambda$ if and only if for all $\mu \in \underline{\mu}$, $\vert\vert \mathbf{z}_\mu \vert\vert^2 = r_\mu^2$. Moreover, the action of $K_\mathrm{M}$ on $\Phi\n(\mathrm{M})$ is transitive.
\end{lemma}

The final reduction concerns the isotropy subgroup. Given $(\underline{\lambda},\underline{\mu})$, define $\tilde p \in \mathcal{O}_\Lambda$ of the form \eqref{m2} such that for all $\mu \in [\underline{\mu}]$, 
\begin{equation}\label{m3}
	\mathbf{z}_\mu = \left( 
	\begin{array}{c}
		r_\mu \\
		0 \\
		\vdots \\
		0
	\end{array}
	\right).
\end{equation}
By construction, $K_{\tilde p}=L_{(\underline{\lambda},\underline{\mu})}$.  The MGS data of every other point $p\in \Phi\n(\mathrm{M})$ is equivalent to that of $\tilde p$ by Lemma \ref{un action is multiplicity free}.

\begin{remark}
	Many of the facts mentioned in this section are also useful for studying Gelfand-Zeitlin systems \cite{GS1,CKO}.
\end{remark}

\subsection{The isotropy representation}\label{section computation of isotropy rep}

Continuing from the previous section, this section computes the isotropy representations at the points $\tilde p \in \Phi\n(\mathrm{M})$ as described in \eqref{m2}, \eqref{m3} and Lemma~\ref{risquare}.

\begin{lemma}\label{form of vector spaces} Let $p \in \Phi\n(\mathrm{M})$ and let $c, \mathbf{z}$ be defined as in \eqref{m2}. The  subspace $T_p(K\cdot p)^\omega$ consists of all matrices of the form 
\begin{equation}\label{form of orthocomplement}
	 \left(\begin{array}{c|c}
		0 & (c-\mathrm{M})\mathbf{x}^\dagger+ \mathbf{z}^\dagger X^\dagger\\
		\hline
		(c-\mathrm{M})\mathbf{x}+ X\mathbf{z} &  0
	\end{array}\right), \quad X \in \mathfrak{k}, \, \mathbf{x} \in M_{n\times 1}(\CC)
\end{equation}
such that
\begin{equation}\label{condition}
	\begin{split}
		0 & = \mathbf{x}^\dagger\mathbf{z}+\mathbf{z}^\dagger\mathbf{x} \\
		0 & = \mathbf{x}\mathbf{z}^\dagger+\mathbf{z}\mathbf{x}^\dagger + [X,\mathrm{M}].
	\end{split}	
\end{equation}
The subspace $T_p(K\cdot p)\cap T_p(K\cdot p)^\omega$  consists of all matrices of the form
\begin{equation}\label{orthocomplement intersection}
	\left(\begin{array}{c|c}
		0 & \mathbf{z}^\dagger Y^\dagger \\
		\hline
		Y\mathbf{z} & 0
	\end{array}\right),\quad  Y \in \kk_\mathrm{M}.
\end{equation}

\end{lemma}

\begin{proof} Denote
\[
	\eta := \left(\begin{array}{c|c}
		0 & 0 \\
		\hline
		0 & Y
	\end{array}\right), \quad 
	\xi := \left(\begin{array}{c|c}
		x_0 & -\mathbf{x}^\dagger \\
		\hline
		\mathbf{x} & X
	\end{array}\right), \quad X,Y \in \mathfrak{k}, \, x_0 \in \sqrt{-1}\R, \, \mathbf{x} \in M_{n\times 1}(\CC).
\]
The tangent space $T_p\mathcal{O}_\Lambda$ consists of elements of the form $[\xi,p]$. Since diagonal elements of $\mathfrak{u}(n+1)$ act trivially, set $x_0=0$. Then elements of $T_p\mathcal{O}_\Lambda$ have block form
\[
	[\xi,p] = \left(\begin{array}{c|c}
		-\mathbf{x}^\dagger\mathbf{z}-\mathbf{z}^\dagger\mathbf{x} & (c-\mathrm{M})\mathbf{x}^\dagger+\mathbf{z}^\dagger X^\dagger \\
		\hline
		(c-\mathrm{M})\mathbf{x}+X\mathbf{z} & \mathbf{x}\mathbf{z}^\dagger+\mathbf{z}\mathbf{x}^\dagger + [X,\mathrm{M}]
	\end{array}\right), \quad X \in \mathfrak{k}, \, \mathbf{x} \in M_{n\times 1}(\CC).
\]
Elements of  $T_p(K\cdot p)$  have block form 
\[
	[\eta,p] = \left(\begin{array}{c|c}
		0 & \mathbf{z}^\dagger Y^\dagger \\
		\hline
		Y\mathbf{z} & [Y,\mathrm{M}]
	\end{array}\right), \quad Y \in \kk.
\]
Recall,
\[
	T_p(K\cdot p)^\omega = \left\{ [\xi,p] \in T_p\mathcal{O}_\Lambda \mid (\omega_\Lambda)_p([\xi,p],[\eta,p]) = 0 \, \forall \,Y \in \kk \right\}.
\]
By \eqref{symplectic form equation},
\begin{equation*}
	\begin{split}
		\sqrt{-1}(\omega_\Lambda)_p([\xi,p],[\eta,p]) & = \Tr\left(p[\xi,\eta]\right) \\
		& =  -\Tr(\mathbf{z}^\dagger Y\mathbf{x}) - \Tr(\mathbf{z}\mathbf{x}^\dagger Y) + \Tr(\mathrm{M}[X,Y]) \\
	& = \Tr(([\mathrm{M},X]-\mathbf{x}\mathbf{z}^\dagger - \mathbf{z}\mathbf{x}^\dagger)Y).
	\end{split}
\end{equation*}
Let $\sqrt{-1}E_{i,i}$, $E_{i,j} - E_{j,i}$, and $\sqrt{-1}(E_{i,j} + E_{j,i})$ be standard basis elements for $\kk$ (where $E_{i,j}$ denotes the matrix whose $i,j$-entry is 1 and all other entries are 0). Plugging these elements in for $Y$ yields a system of equations,
\begin{equation}
	\begin{split}
		0 & = x_i\overline{z}_i + z_i \overline{x}_i \quad \forall i \\
		0 & = (\mu_j - \mu_i) (X_{j,i} + X_{i,j}) - (x_j \overline{z}_i + z_j\overline{x}_i - x_i \overline{z}_j - z_i\overline{x}_j) \quad \forall i\neq j\\
		0 & = (\mu_j - \mu_i) (X_{j,i} - X_{i,j}) - (x_j \overline{z}_i + z_j\overline{x}_i + x_i \overline{z}_j + z_i\overline{x}_j) \quad \forall i\neq j,
	\end{split}	
\end{equation}
(where $X_{i,j}$ denotes the $i,j$ entry of $X$) which in turn is equivalent to the system of equations
\begin{equation}
	\begin{split}
		0 & = x_i\overline{z}_i + z_i \overline{x}_i \quad \forall i \\
		0 & = (\mu_j - \mu_i) X_{j,i} - (x_j \overline{z}_i + z_j\overline{x}_i) \quad \forall i\neq j.\\
	\end{split}	
\end{equation}
This system of equations is equivalent to the system of matrix equations \eqref{condition}. It follows from \eqref{condition} that the block diagonal parts of $[\xi,p] \in T_p(K\cdot p)^\omega$ are zero, so $[\xi,p]$ has the form \eqref{form of orthocomplement} subject to the equations \eqref{condition}. By properties of equivariant moment maps, $T_p(K\cdot p)\cap T_p(K\cdot p)^\omega  = T_p(K_\mathrm{M}\cdot p)$ \cite{GS4}. Elements of $ T_p(K_\mathrm{M}\cdot p)$  have block form of \eqref{orthocomplement intersection}, which completes the proof.
\end{proof}

Equations \eqref{condition} dictate the form of the vectors $(c-\mathrm{M})\mathbf{x}+ X\mathbf{z}$, as the next two lemmas demonstrate.

\begin{lemma}\label{Xz formula lemma}
	Let $p \in \Phi\n(\mathrm{M})$ and let $\mathbf{z}$ be defined as in \eqref{m2}. Let $X\in \kk$ and $\mathbf{x}\in M_{n\times 1}(\CC)$ such that
	\begin{equation}\label{condition 2}
		0  = \mathbf{x}\mathbf{z}^\dagger+\mathbf{z}\mathbf{x}^\dagger + [X,\mathrm{M}].
	\end{equation} 
	If the the component of the interlacing pattern of $(\underline{\lambda},\underline{\mu})$ labelled $\mu$ is not a \mtrap-shape, then 
	\[
		(X\mathbf{z})_\mu = \left(\sum_{\substack{\tau \in [\underline{\mu}]\\ \text{\mtrap-shape}}}\frac{r_\tau^2}{\mu - \tau}\right)\mathbf{x}_\mu.
	\]
\end{lemma}

\begin{proof}
	Let $\mu \neq \nu$ distinct elements of $\underline{\mu}$. Let $X_{\mu,\nu}$, $\mathbf{x}_\mu$, $\mathbf{z}_\mu$, etc.\ denote the corresponding blocks of  $X$, $\mathbf{x}$, and $\mathbf{z}$. By \eqref{condition 2}, the $\mu,\nu$ block of $X$ is given by the formula
\[
	 X_{\mu,\nu} = \frac{1}{\mu - \nu}(\mathbf{x}_\mu\mathbf{z}_\nu^\dagger+\mathbf{z}_\mu\mathbf{x}_\nu^\dagger), \quad \forall \mu\neq \nu.
\]
By Lemma \ref{risquare}, if the component of the interlacing pattern of $(\underline{\lambda},\underline{\mu})$ labelled $\mu$ is not a \mtrap-shape, then $\mathbf{z}_\mu=0$. Thus
\[
	(X\mathbf{z})_\mu = \sum_{\substack{\tau \in [\underline{\mu}]\\ \tau\neq \mu}} X_{\mu,\tau} \mathbf{z}_\tau  = \sum_{\substack{\tau \in [\underline{\mu}]\\ \tau\neq \mu}} \frac{1}{\mu - \tau}\mathbf{x}_\mu \mathbf{z}_\tau^\dagger\mathbf{z}_\tau = \left(\sum_{\substack{\tau \in [\underline{\mu}]\\ \text{\mtrap-shape}}} \frac{||\mathbf{z}_\tau||^2}{\mu - \tau}\right) \mathbf{x}_\mu = \left(\sum_{\substack{\tau \in [\underline{\mu}]\\ \text{\mtrap-shape}}} \frac{r_\tau^2}{\mu - \tau}\right) \mathbf{x}_\mu. \qedhere
\]
\end{proof}

Recall the definition of $C_\mu$ from \eqref{Ci definition}.

\begin{lemma}\label{vanishing of coefficient} Let $p$, $X$, and $\mathbf{x}$ as in Lemma \ref{Xz formula lemma} such that \eqref{condition 2} holds. Assume that the component of the interlacing pattern of $(\underline{\lambda},\underline{\mu})$ labelled $\mu$ is not a \mtrap-shape. Then, $C_\mu = 0$
	if and only if the component of the interlacing pattern of $(\underline{\lambda},\underline{\mu})$ labelled $\mu$ is a \wtrap-shape.
\end{lemma}

\begin{proof} First, note that it is sufficient to prove 
\begin{equation}\label{4.6 eq1}
	\prod_{\substack{\lambda \in [\underline{\lambda}]\\ \text{\wtrap-shape}}}(x- \lambda) = (x- c)\prod_{\substack{\mu \in [\underline{\mu}] \\\text{\mtrap-shape}}}(x- \mu) -\sum_{\substack{\mu \in [\underline{\mu}]\\ \text{\mtrap-shape}}}r_\mu^2\prod_{\substack{\tau \in [\underline{\mu}]\\ \text{\mtrap-shape}\\ \tau \neq \mu}}(x- \tau).
\end{equation}
Indeed, since the component of the interlacing pattern labelled $\mu$ is not a \mtrap-shape, plugging in $x=\mu$ yields
\begin{equation}\label{4.6 eq2}
	\prod_{\substack{\lambda \in [\underline{\lambda}]\\ \text{\wtrap-shape}}}(\mu- \lambda) = \left(\mu- c - \sum_{\substack{\tau \in [\underline{\mu}] \\\text{\mtrap-shape}}}\frac{r_\tau^2}{\mu- \tau}\right) \prod_{\substack{\tau \in [\underline{\mu}]\\ \text{\mtrap-shape}}}(\mu- \tau) = -C_\mu \prod_{\substack{\tau \in [\underline{\mu}]\\ \text{\mtrap-shape}}}(\mu- \tau)
\end{equation}
and the factor 
\begin{equation}\label{4.6 eq3}
	\prod_{\substack{\tau \in [\underline{\mu}]\\ \text{\mtrap-shapes}}}(\mu- \tau)
\end{equation}
is non-zero.

Second, applying Lemma \ref{risquare} ($r_\mu= 0$ when the component labelled $\mu$ is not a \mtrap-shape) and rearranging, observe that
\begin{equation}\label{4.6 eq4}
	\begin{split}
		&(x-c)\prod_{\mu \in [\underline{\mu}]}(x-\mu)^{n_\mu([\underline{\mu}])} - \sum_{\mu\in [\underline{\mu}]} r_\mu^2(x-\mu)^{n_\mu([\underline{\mu}]) -1}\prod_{\substack{\tau \in [\underline{\mu}]\\ \tau \neq \mu}}(x-\tau)^{n_\tau([\underline{\mu}])} \\
		& = (x-c)\prod_{\substack{\mu \in [\underline{\mu}]\\ \text{\mtrap-shape}}}(x-\mu)^{n_\mu([\underline{\mu}])}\prod_{\substack{\tau \in [\underline{\mu}]\\ \text{\wtrap,\parallelogram-shape}}}(x-\tau)^{n_\tau([\underline{\mu}])}\\ & \quad - \sum_{\substack{\mu \in [\underline{\mu}]\\ \text{\mtrap-shape}}}r_\mu^2(x-\mu)^{n_\mu([\underline{\mu}]) -1}\prod_{\substack{\tau \in [\underline{\mu}] \\\text{\mtrap-shape}\\ \tau \neq \mu}}(x-\tau)^{n_\tau([\underline{\mu}])}\prod_{\substack{\tau \in [\underline{\mu}] \\ \text{\wtrap,\parallelogram-shape}}}(x-\tau)^{n_\tau([\underline{\mu}])} \\
		& = \left((x-c)\prod_{\substack{ \mu \in [\underline{\mu}] \\ \text{\mtrap-shape}}}(x-\mu) - \sum_{\substack{\mu\in [\underline{\mu}] \\\text{\mtrap-shape}}}r_\mu^2\prod_{\substack{\tau \in [\underline{\mu}] \\\text{\mtrap-shape}\\ \tau \neq \mu}}(x-\tau)\right)\cdot \prod_{\substack{\tau \in [\underline{\mu}] \\ \text{\mtrap-shape}}}(x-\tau)^{n_\tau([\underline{\mu}]) - 1}\prod_{\substack{\tau \in [\underline{\mu}] \\ \text{\wtrap,\parallelogram-shape}}}(x-\tau)^{n_\tau([\underline{\mu}])}.
	\end{split}
\end{equation}
Then \eqref{4.6 eq1} follows by combining \eqref{4.6 eq4} and \eqref{polynomial2}, which completes the proof.
\end{proof}

For $p\in \Phi\n(\mathrm{M})$, let $V_p\subset \CC^n$ denote the image of injective linear map
\begin{equation}
	T\colon T_p(K \cdot p)^\omega \to \CC^n, \quad \left(\begin{array}{c|c}
		0 & (c-\mathrm{M})\mathbf{x}^\dagger+ \mathbf{z}^\dagger X^\dagger\\
		\hline
		(c-\mathrm{M})\mathbf{x}+ X\mathbf{z} &  0
	\end{array}\right) \mapsto (c-\mathrm{M})\mathbf{x}+ X\mathbf{z}
\end{equation}
and let $U_p \subset V_p$ denote the image of $T_p(K\cdot p)\cap T_p(K\cdot p)^\omega$.  Specialize to the case of $\tilde p$ and recall that $K_{\tilde p} = L_{(\underline{\lambda},\underline{\mu})}$.  The map $T$ is $K_{\tilde p}$-equivariant with respect to the action of $K_{\tilde p}$ on $\CC^n$ as a block-diagonal subgroup of $K=U(n)$ acting by the standard representation.  Decompose $\CC^n = \bigoplus_{i=1}^m \CC^{n_i(\underline{\mu})}$, $m = m(\underline{\mu})$. The subspaces $V_{\tilde p}$ and $U_{\tilde p}$ have the form $\bigoplus_{i=1}^m V_i $ (respectively $\bigoplus_{i=1}^m U_i $) for some subspaces $U_i \subset V_i \subset \CC^{n_i(\underline{\mu})}$. The map $T$ descends to an isomorphism of $K_{\tilde p}$-representations, 
\begin{equation}\label{Wp equivariant isomorphism with VmodU}
	W_{\tilde p} = T_{\tilde p}(K\cdot \tilde p)^\omega/(T_{\tilde p}(K\cdot \tilde p)\cap T_{\tilde p}(K\cdot \tilde p)^\omega) \cong \bigoplus_{i=1}^m V_i/U_i.
\end{equation}
The representation of $K_{\tilde p} = L_1 \times \dots \times L_m$ on the right is given in each component by the inclusion $L_i \subset U(n_i(\underline{\mu}))$ and the standard representation of $ U(n_i(\underline{\mu}))$ on $\CC^{n_i(\underline{\mu})}$. This representation of $L_i$ preserves the subspaces $U_i \subset V_i$ so it induces a representation on $V_i/U_i$.

Recall that if the component of the interlacing pattern labelled $\underline{\mu}_i$ is a \parallelogram-shape, then $L_i = U(n_i(\underline{\mu}))$.  

\begin{proposition}\label{prop VmodU} For all $i = 1, \dots, m$, $m = m(\underline{\mu})$, there is an isomorphism of $L_i$ representations
	\[
		V_i/U_i \cong \left\{\begin{array}{ll}
		\CC^{n_i(\underline{\mu})} & \text{if the component of the interlacing pattern of $(\underline{\lambda},\underline{\mu})$ labelled $\underline{\mu}_i$ is a \parallelogram-shape,}\\
		\{0\} & \text{else,}\\
	\end{array}\right.
	\]
	where $\CC^{n_i(\underline{\mu})}$ denotes the standard representation of $U(n_i(\underline{\mu}))$.
\end{proposition}

\begin{proof}
	In general, 
	\[
		U_i = \{ (Y\mathbf{z})_i  \mid Y \in \kk_{\mathrm{M}} \} = \{ Y_{i,i}\mathbf{z}_i \mid Y_{i,i} \in \mathfrak{u}(n_i(\underline{\mu})) \}.
	\]
	If the component of the interlacing pattern of $(\underline{\lambda},\underline{\mu})$ labelled $\underline{\mu}_i$ is a \mtrap-shape, then,  
	by Lemma \ref{risquare}, $\mathbf{z}_i \neq 0$, so $U_i = \CC^{n_i(\underline{\mu})}$ and $V_i/U_i\cong \{0\}$.  If the component of the interlacing pattern of $(\underline{\lambda},\underline{\mu})$ labelled $\underline{\mu}_i$ is not a \mtrap-shape, then,  
	 $\mathbf{z}_i = 0$, so $U_i = \{0\}$.
	
	It remains to determine the subspace $V_i$ when the component of the interlacing pattern of $(\underline{\lambda},\underline{\mu})$ labelled $\underline{\mu}_i$ is not a \mtrap-shape. 
	In this case, it follows by Lemma \ref{Xz formula lemma} that the block
	\[
		((c-\mathrm{M})\mathbf{x}+ X\mathbf{z})_i = (c-\mathrm{M})\mathbf{x}_i+ (X\mathbf{z})_i = C_i\mathbf{x}_i,
	\]
	where $C_i = C_{\underline{\mu}_i}$ as defined in \eqref{Ci definition}.
	By Lemma \ref{form of vector spaces},  
	\begin{equation}
		\begin{split}
			V_i & =\{ ((c-\mathrm{M})\mathbf{x}+ X\mathbf{z})_i \mid X \in \kk,  \mathbf{x} \in \CC^n, \mathbf{x}\mathbf{z}^\dagger+\mathbf{z}\mathbf{x}^\dagger + [X,\mathrm{M}] \} \\
			& = \{ C_i\mathbf{x}_i \mid  \mathbf{x}_i \in \CC^{n_i(\underline{\mu})} \}.
		\end{split}
	\end{equation}
	By Lemma \ref{vanishing of coefficient},  $C_i = 0$ if and only if the component of the interlacing pattern of $(\underline{\lambda},\underline{\mu})$ labelled $\underline{\mu}_i$ is a \wtrap-shape. This completes the proof.
\end{proof}

Thus $\bigoplus_{i=1}^m V_i/U_i$ is isomorphic to the $L_{(\underline{\lambda},\underline{\mu})}$-representation $W_{(\underline{\lambda},\underline{\mu})}$.

\begin{proposition}\label{Wp symplectic structure}  The linear symplectic structure on $W_{(\underline{\lambda},\underline{\mu})}$ defined via the symplectic form $\overline{\omega}_{\tilde p}$ and the isomorphism \eqref{Wp equivariant isomorphism with VmodU} equals the linear symplectic form $
	\omega_{(\underline{\lambda},\underline{\mu})}$ defined in \eqref{omegalmu definition}.

\end{proposition}

\begin{proof} Denote
\[
	\eta := \left(\begin{array}{c|c}
		0 & -\mathbf{y}^\dagger \\
		\hline
		\mathbf{y} & Y
	\end{array}\right), \quad 
	\xi := \left(\begin{array}{c|c}
		0 & -\mathbf{x}^\dagger \\
		\hline
		\mathbf{x} & X
	\end{array}\right), \quad X,Y \in \mathfrak{k}, \, \mathbf{x},\mathbf{y} \in M_{n\times 1}(\CC).
\]
Then, using Lemma \ref{Xz formula lemma},
\begin{equation}\label{symplectic form proof eq 1}
	\begin{split}
		\sqrt{-1}(\omega_\Lambda)_{\tilde p}([\xi,{\tilde p}],[\eta,{\tilde p}]) & = \Tr\left({\tilde p}[\xi,\eta]\right)	 = \Tr\left([{\tilde p},\xi]\eta\right)	\\
		& = -\Tr\left(\left(\begin{array}{c|c}
			0 & (c-\mathrm{M})\mathbf{x}^\dagger+ \mathbf{z}^\dagger X^\dagger\\
			\hline
			(c-\mathrm{M})\mathbf{x}+ X\mathbf{z} &  0
			\end{array}\right)\left(\begin{array}{c|c}
			0 & -\mathbf{y}^\dagger \\
			\hline
			\mathbf{y} & Y
			\end{array}\right)\right) \\
		&  = -((c-\mathrm{M})\mathbf{x}^\dagger+ \mathbf{z}^\dagger X^\dagger)\mathbf{y} + \Tr(((c-\mathrm{M})\mathbf{x}+ X\mathbf{z})\mathbf{y}^\dagger) \\
		&  = -((c-\mathrm{M})\mathbf{x}^\dagger+ \mathbf{z}^\dagger X^\dagger)\mathbf{y} + \Tr(\mathbf{y}^\dagger((c-\mathrm{M})\mathbf{x}+ X\mathbf{z})) \\
		&  = -(c-\mathrm{M})(\mathbf{x}^\dagger\mathbf{y} - \mathbf{y}^\dagger\mathbf{x}) - \mathbf{z}^\dagger X^\dagger\mathbf{y} +  \mathbf{y}^\dagger X\mathbf{z} \\
		&  = -(c-\mathrm{M})(\mathbf{x}^\dagger\mathbf{y} - \mathbf{y}^\dagger\mathbf{x}) - ( X\mathbf{z})^\dagger\mathbf{y} +  \mathbf{y}^\dagger X\mathbf{z} \\
		&  = -(c-\mathrm{M})(\mathbf{x}^\dagger\mathbf{y} - \mathbf{y}^\dagger\mathbf{x}) +\sum_{i=1}^m\left(\sum_{\substack{\text{\mtrap-shape}\\ j\neq i}}\frac{r_j^2}{\mu_i - \mu_j}\right)(-\mathbf{x}_i^\dagger \mathbf{y}_i + \mathbf{y}_i^\dagger\mathbf{x}_i). \\
	\end{split}
\end{equation}
Viewing $[\xi,{\tilde p}]$ and $[\eta,{\tilde p}]$ as representatives of vectors in the isotropy representation,
\begin{equation}\label{symplectic form proof eq 3}
	\begin{split}
		(\overline{\omega}_\lambda)_{\tilde p}([\xi,{\tilde p}],[\eta,{\tilde p}]) 
		& = \frac{1}{\sqrt{-1}}\sum_{\substack{\text{\parallelogram-shape}\\ i=1}}^m\left(c - \mu_i +\sum_{\substack{\text{\mtrap-shape}\\ j\neq i}}\frac{r_j^2}{\mu_i - \mu_j}\right)(-\mathbf{x}_i^\dagger \mathbf{y}_i + \mathbf{y}_i^\dagger\mathbf{x}_i)\\
		& = \frac{1}{\sqrt{-1}}\sum_{\substack{\text{\parallelogram-shape}\\ i=1}}^mC_i(-\mathbf{x}_i^\dagger \mathbf{y}_i + \mathbf{y}_i^\dagger\mathbf{x}_i).
	\end{split}
\end{equation}
Applying the isomorphism $T\colon W_{\tilde p}\to W_{(\underline{\lambda},\underline{\mu})}$, $[\xi,p]\mapsto \mathbf{u} = (C_i\mathbf{x}_i)_i $, $[\eta,p]\mapsto \mathbf{v} = (C_i\mathbf{y}_i)_i $  yields
\[
	\omega_{(\underline{\lambda},\underline{\mu})}(\mathbf{u},\mathbf{w})  = \frac{1}{\sqrt{-1}}\sum_{\substack{\text{\parallelogram-shape}\\ i=1}}^m\frac{-\mathbf{u}_i^\dagger \mathbf{w}_i + \mathbf{w}_i^\dagger\mathbf{u}_i}{C_i}. \qedhere
\]
\end{proof}

\bibliographystyle{alpha}
\bibliography{../../../math_general/master_bibliography.bib}

\begin{thebibliography}{CKO20}

\bibitem[ACK18]{ACK}
B.~H. An, Y.~Cho, and J.~S. Kim.
\newblock On the {$f$}-vectors of {G}elfand-{C}etlin polytopes.
\newblock {\em European J.~Combin.}, 67:61--77, 2018.

\bibitem[Ala09]{AL}
I.~Alamiddine.
\newblock {\em G\'eom\'etrie de syst\`emes hamiltoniens int\'egrables: Le cas
  du syst\`eme de {G}elfand-{C}etlin}.
\newblock PhD thesis, Universit\'e Toulouse, 2009.

\bibitem[ALL18]{ALL}
A.~Alekseev, J.~Lane, and Y.~Li.
\newblock The {$U(n)$} {G}elfand-{Z}eitlin system as a tropical limit of
  {G}inzburg-{W}einstein diffeomorphisms.
\newblock {\em Philosophical Transactions of the Royal Society A: Mathematical,
  Physical and Engineering Sciences}, 376, 2018.

\bibitem[Aud04]{audin}
Mich{\`e}le Audin.
\newblock {\em Torus actions on symplectic manifolds}, volume~93 of {\em
  Progress in Mathematics}.
\newblock Birkh\"auser Verlag, revised edition, 2004.

\bibitem[BMZ18]{BMZ}
D.~Bouloc, E.~Miranda, and N.~T. Zung.
\newblock Singular fibers of the {G}elfand-{C}etlin system on
  {$\mathfrak{u}(n)^*$}.
\newblock {\em Philosophical Transactions of the Royal Society A: Mathematical,
  Physical and Engineering Sciences}, 376, 2018.

\bibitem[CKO20]{CKO}
Yunhyung Cho, Yoosik Kim, and Yong-Geun Oh.
\newblock Lagrangian fibers of {G}elfand-{C}etlin systems.
\newblock {\em Advances in Mathematics}, 372, 2020.

\bibitem[GS83]{GS1}
V.~Guillemin and S.~Sternberg.
\newblock The {G}elfand-{C}etlin system and quantization of the complex flag
  manifolds.
\newblock {\em J.\ Funct.\ Anal.}, 52(1):106--128, 1983.

\bibitem[GS84a]{GS3}
V.~Guillemin and S.~Sternberg.
\newblock Multiplicity-free spaces.
\newblock {\em J.\ Differential Geom.}, 19(1):31--56, 1984.

\bibitem[GS84b]{GS4}
V.~Guillemin and S.~Sternberg.
\newblock {\em Symplectic techniques in physics}.
\newblock Cambridge University Press, Cambridge, 1984.

\bibitem[GS84c]{GS5}
Victor Guillemin and Shlomo Sternberg.
\newblock A normal form for the moment map.
\newblock In Shlomo Sternberg, editor, {\em Differential geometric methods in
  mathematical physics ({J}erusalem, 1982)}, volume~6, pages 161--175. Reidel,
  1984.

\bibitem[Kir84]{kirwan}
F.~Kirwan.
\newblock Convexity properties of the moment mapping, {III}.
\newblock {\em Inventiones Mathematicae}, 77(3):547--552, 1984.

\bibitem[Kno10]{knop1}
F.~Knop.
\newblock Automorphisms of multiplicity free {H}amiltonian manifolds.
\newblock {\em J. Amer. Math. Soc.}, 24(2):567--601, 2010.

\bibitem[Lan18]{L2}
J.~Lane.
\newblock The geometric structure of symplectic contraction.
\newblock {\em International Mathematics Research Notices}, 2018.

\bibitem[Los09]{losev}
I.~V. Losev.
\newblock Proof of the {K}nop conjecture.
\newblock {\em Ann. Inst. Fourier (Grenoble)}, 59(3):1105--1134, 2009.

\bibitem[Mar85]{MARL}
C.-M. Marle.
\newblock Mod\`ele d'action {H}amiltonienne d'un groupe de {L}ie sur une
  vari\'et\'e symplectique.
\newblock {\em Rend. Sem. Mat. Univ. Politec. Torino}, 43(2):227--251, 1985.

\bibitem[Pab14]{P}
M.~Pabiniak.
\newblock Gromov width of non-regular coadjoint orbits of {$U(n)$}, {$SO(2n)$}
  and {$SO(2n+1)$}.
\newblock {\em Mathematical Research Letters}, 21(1):187--205, 2014.

\end{thebibliography}

\Addresses

\end{document}